\documentclass[9pt]{article}

\usepackage{cite}
\usepackage{amsmath,amssymb,amsfonts}
\usepackage{algorithmic}
\usepackage{graphicx}
\usepackage{textcomp}
\usepackage{amsthm}

\usepackage{enumerate}
\usepackage{xspace}
\usepackage{tikz-cd}

\newcommand{\lms}{$\left(\mathcal{LMS}\right)$\xspace}
\newcommand{\ms}{$\left(\mathcal{MS}\right)$\xspace}
\newcommand{\T}{\mathrm{T}}

\newcommand{\rk}{\mathrm{rank \ }}

\newcommand{\fp}[2]{\frac{\partial #1}{\partial #2}}
\newcommand{\spn}[1]{\mathrm{span}\left\{#1\right\}}
\newcommand{\ti}[1]{\tilde{#1}}

\usepackage{authblk}

\newtheorem{definition}{Definition}
\newtheorem{theorem}{Theorem}

\newtheorem{lemma}{Lemma}
\newtheorem{remark}{Remark}

\begin{document}
\title{Mechanical feedback linearization of single-input mechanical control systems}
\author[1]{Marcin Nowicki}
\author[2,3]{ Witold Respondek}
\affil[1]{Poznan University of Technology, Institute of Automatic Control and Robotics, Piotrowo 3a, 61-138~Pozna\'n, Poland}  
\affil[2]{Lodz University of Technology, Institute of Automatic Control, B. Stefanowskiego 18,  90-537 Lodz, Poland }  
\affil[3]{INSA de Rouen Normandie, Laboratoire de Math\'ematiques, 76801 Saint-Etienne-du-Rouvray, France}

\maketitle
\begin{abstract}
	We present a new type of feedback linearization that is tailored for mechanical control systems. We {call} it a mechanical feedback linearization. 
Its basic feature is preservation of the mechanical structure of the system. 
For mechanical systems with a scalar control, we formulate necessary and sufficient conditions that are verifiable using differentiations and algebraic operations only. We illustrate our results {with} several examples.
\end{abstract}

\section{Introduction}
\label{sec:introduction}
An $N$-dimensional control-affine system with a scalar control
\begin{align}
	\label{eq:cs}
	\tag{$\Sigma$}
	\dot{z}=F(z)+G(z)u,
\end{align}
where $z\in Z$, an open subset of $\mathbb{R}^N$, and $u\in \mathbb{R}$, is said to be (locally) \textit{feedback linearizable} {(F-linearizable)} if there exist a (local) diffeomorphism $\Phi: \ Z\rightarrow \mathbb{R}^N$ and an invertible feedback of the form $u=\alpha(z)+\beta(z)\tilde{u}$ such that the control system \eqref{eq:cs}, in  the new coordinates ${\ti{z}=\Phi(z)}$ and with the new control $\tilde{u}$, is a controllable linear system of the form ${\dot{\ti{z}}=A\ti{z}+b\ti{u}}$. A geometric solution {to} the problem of feedback linearization (inspired by \cite{brockett}, and developed independently in \cite{JR} and \cite{HS}) provides powerful techniques for designing a closed-loop control system that have been used in numerous engineering applications. From a theoretical point of view, that result identifies a class of nonlinear systems that can be considered as linear ones in a well-chosen coordinates and with respect to a well-modified control.

In this paper, we state and study the following fundamental question: if a nonlinear control system \eqref{eq:cs} is mechanical and feedback linearizable, are those two structures compatible? That is, can we feedback linearize the system preserving its mechanical structure? For mechanical control systems, it is natural to consider mechanical feedback equivalence (in particular, to a linear form) under {mechanical} transformations (coordinates changes and feedback) that preserve the mechanical structure of the system. In our recent paper \cite{NR}, we showed that even in the simplest underactuated case of 2 degrees of freedom, the structures (linear and mechanical) may not conform trivially. In the present paper, we treat the single-input case in its full generality.

{
There are several motivations for preserving the mechanical structure when feedback linearizing the system. 
First, our formulation of the problem of mechanical linearization preserves configurations and velocities. We reckon that it is essential that new configurations (of the linearized system) are functions of the original configurations only, as well as new velocities are true physical velocities (in contrast to pseudo-velocities). Therefore, we do not lose the physical interpretation of the system. This could be useful, e.g. for mechanical systems with constraints on configurations, which are transformed into linear constraints on configurations.
Second, the configuration trajectories are preserved too, which could be useful in e.g. the motion planning problem (the most natural way to state the problem for mechanical systems is to follow configuration trajectories).
Third, it is worth mentioning that mechanical feedback linearizability  guarantees the linearizing outputs to be functions of configurations only. This may be of constructional importance because one needs only configuration sensors, not those of velocities.
The next argument is the fact that the resultant linear mechanical system allows us to employ dedicated techniques for mechanical systems. An example of such technique is the natural frequency method of tuning a linear feedback.
Finally, when applying mechanical feedback linearization, the physical interpretation of the external action (force, torque, etc.) is preserved but is lost for general feedback linearization. }

This work is a mechanical counterpart of the classical results on feedback linearization of control systems \cite{brockett}, \cite{JR}, \cite{HS}, see also monographs \cite{nij}, \cite{Isidori}. Our intention is to formulate conditions {for mechanical linearization (shortly, MF-linearization)} in a possibly similar manner (e.g. using involutivity of certain distributions).

For a geometric approach to mechanical control systems see \cite{bullo}, \cite{bloch}, \cite{rr}, \cite{rr3}.  For mathematical preliminaries concerning the Lie derivative, the Lie bracket, distributions, etc., see \cite{nij}, \cite{Isidori}. For linearization of mechanical control systems along controlled trajectories see \cite{bullo2}. For mechanical state-space linearization of mechanical control systems see \cite{rr2} and \cite{lh}. Compare also \cite{spong}, for a pioneering work on (partial) feedback linearization of mechanical systems.

Although the state-space of {mechanical control system} is the tangent bundle $\T Q$ of the configuration space $Q$, we formulate our conditions using objects on $Q$ only. The key here is a {geometric} approach to mechanical systems \cite{bullo} and considering {the} Euler-Lagrange equations as the geodesic equation under an influence of external forces.  

{The outline of the paper is as follows}. In Section \ref{sec:PS}, we state the problem. In Section \ref{sec:MF}, we develop further the problem of mechanical feedback linearization and formulate the main result, {separately, for mechanical systems with $n\geq3$ in Theorem \ref{thm:MFC}, and with $n=2$ in Theorem \ref{thm:MFn2}}. In Section \ref{sec:app}, we provide an application of our results to MF-linearization of several mechanical systems.  Section \ref{sec:apx} contains technical results used in proofs that could be of independent interest.

\subsection{Notation}
{Throughout the Einstein summation convention is assumed, i.e. any expression containing a repeated index (upper and lower) implies the summation over that index up to $n$, e.g. $\omega_iX^i=\sum_{i=1}^{n}\omega_iX^i$.\\
$A^T$ \quad transpose of a matrix (of a vector) $A$,\\
$I_n$ \quad $n\times n$ identity matrix,\\
$Q$ \quad configuration manifold, \\ 
$\mathfrak{X}(Q)$ \quad the set of smooth vector fields on a manifold $Q$,\\
$\T_xQ$ \quad tangent space at $x\in Q$, \\
$\T Q=\bigcup_{x\in Q}T_xQ $\quad tangent bundle of $Q$, \\
$x=(x^1,\ldots,x^n)$ \quad a local coordinate system on $Q$, \\ 
$\phi$ \quad a diffeomorphism of $Q$,  and $\Phi$  a diffeomorphism of $\T Q$, \\
$D\phi=\fp{\phi}{x}$ \quad the Jacobian matrix of a diffeomorphism $\phi$, \\
$\fp{\tilde{x}^i}{x^j}:=\fp{\phi_i}{x^j}$\quad the $(i,j)$-element of the Jacobian matrix $D\phi$, \\ 
$\fp{x^j}{\tilde{x}^i}$ \quad the $(j,i)$-element of the inverse of the Jacobian matrix $D\phi$,  \\
$L_X\alpha$ \quad Lie derivative of a function $\alpha$ defined as $L_X\alpha=\fp{\alpha}{x^i}X^i$, \\
$\left[X,Y\right]=\fp{Y}{x}X-\fp{X}{x}Y= ad_XY$ \quad  Lie bracket of vector fields, \\ 
$\fp{ }{x^i}$ \quad the $i$-th unity vector field, and $dx^i$ the $i$-th unity covector field, in a coordinate system $x=(x^1,\ldots,x^n ) $,\\
$	\mathcal{E}^i=\spn{ad_e^jg,\ 0\leq j\leq i }$  distribution on $Q$ spanned by $ad_e^jg$, \\
$\nabla$ \quad covariant derivative, and $\nabla^2$  second covariant derivative,\\
$\Gamma^i_{jk}$ \quad Christoffel symbols of the second kind of $\nabla$, \\
}

\section{Problem statement}
\label{sec:PS}
{Consider an $n$-dimensional configuration space $Q$ (an open subset of $\mathbb{R}^n$ or, in general, an $n$-dimensional manifold) equipped with a symmetric affine connection $\nabla$.  The operator of the affine connection $\nabla$ allows to define intrinsically the acceleration as the covariant derivative $\nabla_{\dot{x}(t)}\dot{x}(t)$, see e.g. \cite{bullo,bloch,lee}. The covariant derivative
	$
	\nabla: \mathfrak{X}(Q) \times \mathfrak{X}(Q) \rightarrow \mathfrak{X}(Q)
	$
 of an arbitrary vector field $Y=Y^i\fp{}{x^i}$ with respect to $X=X^i\fp{}{x^i}$ in coordinates reads
	\begin{align}
		\label{eq:CD1}
		\nabla_XY=\left(\fp{Y^i}{x^j}X^j+\Gamma^i_{jk}X^jY^k \right) \fp{ }{x^i}.
\end{align}}
A mechanical control system \ms is a 4-tuple $(Q,\nabla,g,e)$, where  $g$ and $e$ are, respectively, controlled and uncontrolled vector fields on $Q$.
A curve $x(t):I\rightarrow Q$, $I\subset \mathbb{R}$, is a trajectory of \ms if it satisfies the following equation
\begin{align}
	\nabla_{\dot{x}(t)}\dot{x}(t)=e\left(x(t) \right) +g\left( x(t)\right) u,
	\label{eq:ms0}
\end{align}
which can be viewed as an equation that balances accelerations of the system, where the left{-}hand side represents geometric accelerations (i.e. accelerations caused by the geometry of the system) and the right{-}hand side represents accelerations caused by external actions on the system (controlled or not). 
{Notice that \eqref{eq:ms0} is a second-order differential equation on $Q$ (indeed, using \eqref{eq:CD1} we conclude that $\nabla_{\dot{x}}\dot{x}$ depends on $\ddot{x}$, see \cite{bullo} for details) and can be rewritten as a system of} first-order differential equations on $\T Q$, which we also call a mechanical control system \ms:
\begin{align}
	\begin{split}
		\dot{x}^i&=y^i\nonumber\\
		\dot{y}^i&=-\Gamma^i_{jk}(x)y^j y^k+e^i(x)+g^i(x) u,
	\end{split}
	\label{eq:ms}
	\tag{$\mathcal{MS}$}
\end{align}
for $1\leq i \leq n$, where $\left( x,y\right) =\left(x^1,\ldots,x^n,y^1,\ldots,y^n \right) $ are local coordinates on the tangent bundle $\T Q$ of the configuration manifold $Q$, and $\Gamma^i_{jk}(x)$ are Christoffel symbols of the affine connection $\nabla$ that correspond to the Coriolis and centrifugal forces. The vector fields $e(x)=( e^1(x),\ldots,e^n(x))^T$ and $g(x)=(g^1(x),\ldots,g^n(x))^T$ correspond to, respectively, uncontrolled and controlled actions on the system. Throughout all objects are assumed to be smooth and the word smooth means $C^{\infty}$-smooth.

{Our obvious inspirations are Lagrangian mechanical control systems without dissipative forces. For the correspondence between \eqref{eq:ms} and the Lagrangian equations of dynamics see \cite{bullo}, \cite{bloch}, \cite{rr} and our recent papers \cite{lh}, \cite{lms}. However, we will consider throughout a more general class of mechanical control systems allowing for any symmetric (not necessarily a metric) connection and any (not necessarily potential) vector field $e(x)$.}

 Consider the group of mechanical feedback transformations {$G_{MF}$} generated by the following transformations:
\begin{enumerate}[(i)]
	\item changes of coordinates in $\T Q$ given by $\Phi:\  \T Q \rightarrow  \T \tilde{Q}$
	\begin{align}
		\label{eq:diff}
		\begin{split}
			\left(x,y \right)  &\mapsto \left( \ti{x},\ti{y}\right) =\Phi(x,y)=\left( \phi(x),\fp{\phi}{x}(x)y \right), 
		\end{split}
	\end{align}
called a mechanical diffeomorphism,	where $\phi:Q\rightarrow \ti{Q}$ is a diffeomorphism and $\fp{\phi}{x}$ its Jacobian matrix,
	\item mechanical feedback transformations, denoted $\left(\alpha,\beta,\gamma \right) $, of the form
	\begin{align}
		\label{eq:feed}
		u=\gamma_{jk}(x)y^jy^k+\alpha(x)+\beta(x)\tilde{u},
	\end{align}
	where $\gamma_{jk},\alpha,\beta$ are smooth functions on $Q$ satisfying \\$\gamma_{jk}=\gamma_{kj}$, $\beta(\cdot)\neq 0$. {The matrix $\gamma=(\gamma_{jk})$ represents a $(0,2)-$tensor field.}
\end{enumerate}
Even if the diffeomorphism $\phi$ is possibly local on $Q$, {the action of $\fp{\phi}{x}(x)$} is always global on fibers $T_xQ$.
\begin{definition}
	\label{def:lin}
The system \eqref{eq:ms} is {MF-linearizable} if there exist mechanical feedback transformations $(\Phi,\alpha,\beta,\gamma)\in {G_{MF}}$ bringing \eqref{eq:ms} into a linear controllable mechanical system of the form
\begin{align}
	\begin{split}
		\dot{\ti{x}}^i&=\ti{y}^i\\
		\dot{\ti{y}}^i&=E^i_j\ti{x}^j+b^i\ti{u},
	\end{split}
	\label{eq:lms}
	\tag{$\mathcal{LMS}$}
\end{align}
where $(\ti{x},\ti{y})$ are coordinates on $\T\mathbb{R}^n=\mathbb{R}^n\times\mathbb{R}^n$, the matrix $E=(E^i_j) $ is  an $n\times n$ real-valued matrix, the vector field $b=b^i \fp{}{\ti{x}^i}$ is constant, and the pair $\left(E,b \right) $ is controllable (see \cite{Hughes}). 
\end{definition}

 Represent \eqref{eq:ms} as $\dot{z}=F(z)+G(z)u,$ where $z=(x,y)\in \T Q$, $F=y^i \frac{\partial}{\partial x^i}+\left(-\Gamma^i_{jk}(x)y^j y^k+e^i(x)\right)\frac{\partial}{\partial y^i}$, and $G=g^i(x)\frac{\partial}{\partial y^i}$. 
 {The problem that we formulate and solve in the paper is whether \ms is MF-linearizable? That is, {do} there exist $\Phi=(\ti{x},\ti{y})=(\phi, \fp{\phi}{x} y)$ and $(\alpha,\beta,\gamma)$ such that } 
 	\begin{align*}
 \quad 	\fp{\Phi}{z}(z)\left( F+G(y^T\gamma y+\alpha)\right)(z) &=\left( \begin{matrix}
 		\ti{y}\\ E\tilde{x}
 	\end{matrix}\right), \\ \fp{\Phi}{z}(z)\left(  G\beta\right)(z)&=\left(\begin{matrix}
 0\\b
 \end{matrix}\right)?
 \end{align*}

{Note that} MF-linearizability is stronger than the classical feedback linearizability since, for the latter, $\Phi: \T Q \rightarrow \mathbb{R}^{2n}$ can be any diffeomorphism (need not be of mechanical form \eqref{eq:diff}) and $y^T\gamma(x) y+\alpha(x)$ can be replaced by any function $\alpha(x,y)$ on $\T Q$ and $\beta(x)$ by any invertible function $\beta(x,y)$ on $\T Q$. 

{If we neglect the mechanical structure of $\dot{z}=F(z)+G(z)u$, and consider it as a general control system, we can ask if the system is F-linearizable. The well-known answer{\cite{JR,HS}} asserts that{, locally,} this is the case  if and only if the distributions $\mathcal{D}^i=\spn{ad_F^jG, 0\leq j\leq i }$ are involutive and of constant rank for $i=0,...,2n-1$ {and $\mathcal{D}^{2n-1}=\T Q$}. The natural question arises whether, for F-linearizable \ms, the general feedback transformations $\left( \Phi(z),\alpha(z),\beta(z)\right)$ are mechanical (i.e. of the form \eqref{eq:diff} and \eqref{eq:feed}) or whether they can be replaced by mechanical ones.}

\textbf{Example 1:}
{
Consider the mechanical system
\begin{align}
	\label{eq:ex1}
	\begin{split}
		\dot{{x}}^{1}&={y}^{1}\\
		\dot{{x}}^{2}&={y}^{2}
	\end{split}
	\quad
	\begin{split}
		\dot{{y}}^{1}&=-x^2(y^1)^2+{x}^2\\
		\dot{{y}}^{2}&=u,
	\end{split}
\end{align}
on $\mathbb{R}^4$. This system is locally F-linearizable. Indeed, the local diffeomorphism $\ti{z}=\Phi(z)$, where $z=(x^1,x^2,y^1,y^2)$, $\ti z=(\ti x^1,\ti x^2,\ti y^1,\ti y^2)$, given by
\begin{align*}
	\begin{split}
		\ti{x}^1&={x}^{1}\\
		\ti{{x}}^{2}&=x^2-x^2({y}^{1})^2
	\end{split}
	\quad
	\begin{split}
		\ti{{y}}^{1}&=y^1\\
		\ti{{y}}^{2}&=\left((y^1)^2-1 \right) \left(2(x^2)^2y^1-y^2 \right),
	\end{split}
\end{align*}
together with the feedback $u=2(x^2)^3+6 (x^2-(x^2)^2)(y^1)^2 + \frac{\ti{u}}{(y^1)^2-1} $, render the original system linear and controllable
\begin{align*}
	\dot{\ti{x}}^{1}=\ti{y}^{1} \quad \dot{\ti{x}}^{2}=\ti{y}^{2} \quad \dot{\ti{y}}^{1}=\ti{x}^2 \quad \dot{\ti{y}}^{2}=\ti{u}.
\end{align*}
Therefore, the system is F-linearizable. Note, however, that neither the change of coordinates nor the feedback is mechanical ($\ti x^2$ depends on velocities, and the function $\beta$ depends on velocities as well) so the mechanical structure is not preserved. Our question is whether this system can be linearized by other transformations that preserve the mechanical structure, i.e. can it be MF-linearized?
}

The group of mechanical transformations ${G_{MF}}=\left\lbrace \left(\Phi,\alpha,\beta,\gamma \right) \right\rbrace $ preserves trajectories, that is, maps the trajectories of \ms into those of its MF-equivalent system $(\widetilde{\mathcal{MS}})$. Indeed, if $z\left( t,z_0,u(t)\right) $ is a trajectory of \eqref{eq:ms} (passing through $z_0=(x_0,y_0)$ and corresponding to a control $u(t)$), then $\tilde{z}\left( t,\tilde{z}_0,\tilde{u}(t)\right) =\Phi\left(z\left( t,z_0,u(t)\right)  \right)$ is a trajectory of $(\widetilde{\mathcal{MS}})$ passing through $\ti{z}_0~=~\Phi(z_0)=(\phi(x_0), \fp{\phi}{x}(x_0)y_0)$ and corresponding to $\ti{u}(t)$, where ${u(t)=y(t)^T\gamma\left(x(t)\right)y(t)+\alpha\left( x(t)\right) + \beta \left(x(t) \right) \ti{u}(t)} $.

\noindent Moreover, via $\phi:Q\rightarrow \ti{Q}$, it establishes a correspondence between configuration trajectories in $Q$ and $\ti{Q}$, i.e. $\ti{x}\left(t,\ti{z}_0,\ti{u}(t)\right)=\phi\left(x(t,z_0,u(t)) \right)$,
making the following diagram commutative (notice, however, that $\pi\left(z(t,z_0,u) \right)=x(t,z_0,u) $ depends on $z_0=(x_0,y_0)$, {i.e.} an initial configuration $x_0$ \textit{and} initial velocity $y_0$):
\[
\begin{tikzcd}[row sep = 1cm,column sep=3cm]
	z(t,z_0,u) \arrow{r}{ \left( \Phi,\alpha,\beta,\gamma\right)  }\arrow{d}{\pi} & \ti{z}(t,\ti{z}_0,\ti{u}) \arrow{d}{\pi} \\
	x(t,z_0,u) \arrow{r}{ \left( \phi,\alpha,\beta,\gamma\right)  }  & \ti{x}(t,\ti{z}_0,\ti{u}) 
\end{tikzcd}
\]
where $\pi:TQ\rightarrow Q$, $\pi(z)=\pi(x,y)=x$, is the canonical projection which assigns to the pair $(x,y)$ the point $x$ at which the velocity $y$ is attached.

\section{Mechanical feedback linearization}
\label{sec:MF}
Our main result uses two basic ingredients: the covariant derivative of the connection $\nabla$, see \eqref{eq:CD1}, and the involutivity of suitable distributions. We will also need the second covariant derivative of a vector field $Z$ in the directions $(X,Y)$, which is a mapping 
	\begin{align}
		\label{eq:cd2}
		\begin{split}
				\nabla^2: \mathfrak{X}(Q)\times\mathfrak{X}(Q)\times\mathfrak{X}(Q)\rightarrow\mathfrak{X}(Q)\\
			\nabla^2_{X,Y}Z= \nabla_X\nabla_YZ - \nabla_{\nabla_XY}Z.
		\end{split}
	\end{align}
For properties of the second covariant derivative see Lemma \ref{prop:cov2} in Appendix.

In order to formulate the result, we associate with \ms the following sequence of nested distributions $\mathcal{E}^0 \subset \mathcal{E}^1 \subset \mathcal{E}^2 \subset \ldots \subset \mathcal{E}^i \subset \ldots \subset \T Q$, where
\begin{align*}
	\mathcal{E}^0=\spn{g}, \quad
	\mathcal{E}^i=\spn{ad_e^jg, 0\leq j\leq i }.
\end{align*}
\begin{remark}
	\label{rem:inv}
	To analyze the behavior of the distributions $\mathcal{E}^i$ under mechanical feedback transformations $(\alpha,\beta,\gamma)$ notice, first, that $\mathcal{E}^i$ are invariant under $\gamma$ since $\gamma$ does not act on them. If the distributions $\mathcal{E}^i$ are involutive, then they are invariant under feedback transformations of the form $(\alpha,\beta,0)$, i.e. for $\gamma=0$ they remain unchanged if we replaced $e$ and $g$ by, respectively, $e+g\alpha$ and $\beta g$, cf. \cite{nij}, \cite{Isidori}.
\end{remark}

Now, we formulate our main result {for} MF-linearization. { First, we state a theorem for \ms with $n\geq3$ degrees of freedom. The remaining case of $n=2$ degrees of freedom is treated in Theorem~\ref{thm:MFn2}. For an explanation of that distinction, see the comment before Theorem \ref{thm:MFn2} and Remark \ref{rem:2d} for a comparison of both results.}

{By a local MF-linearization around $x_0\in Q$ we mean that it holds on $\bigcup_{x\in \mathcal{O}}T_xQ$, where $\mathcal{O}$ is a neighborhood of $x_0$; recall that all transformations are global on tangent spaces $T_xQ$.   }

\begin{theorem}
	\label{thm:MFC}
	Assume $n \geq 3$. A mechanical control system $(\mathcal{MS})$ is, locally around $x_0$, MF-linearizable to a controllable $(\mathcal{LMS})$ if and only if
	\begin{enumerate}[(MF1)]
		\item $\rk \mathcal{E}^{n-1} =n$,
		\item $\mathcal{E}^{i}$ is involutive and of constant rank, for $0\leq i \leq n-2$,
		\item $\nabla_{ad_e^ig}\, g\in \mathcal{E}^0$\quad for $0\leq i\leq n-1$,
		\item $\nabla^2_{ad_e^kg,ad_e^jg} \, e\in \mathcal{E}^1$\quad for $0\leq k,j\leq n-1$,
	\end{enumerate}
\end{theorem}

\begin{remark}
	{
	Notice that (MF1)-(MF2) are the classical conditions (see \cite{JR,HS,Isidori,nij}) that assure F-linearization of the system $\dot{x}=e(x)+g(x)u$ on $Q$ via $\ti{x}=\phi(x)$ and $u=\alpha(x)+\beta(x)\ti{u}$. The remaining two, (MF3)-(MF4), can be interpreted as compatibility conditions that guarantee vanishing the Christoffel symbols $\Gamma^i_{jk}$ in the linearizing coordinates $\ti{x}=\phi(x)$, except for those that can be compensated by feedback 	$u=\gamma_{jk}(x)y^jy^k+\tilde{u}$. }
\end{remark}

\begin{proof}
	In the proof we will use two Lemmata \ref{prop:cov2} and \ref{lem:formula}, given in Appendix, that are of independent interest.
	
	\textit{Necessity}.	For \eqref{eq:lms}, we have $\Gamma^i_{jk}=0$, $e=Ex$ and $g=b$. It follows that $ad_e^ig=(-1)^iE^ib$ and therefore, using  the definitions of $\nabla$, given by \eqref{eq:CD1}, and of  $\nabla^2$, given by \eqref{eq:cd2}, we calculate 
	\begin{align}
		\label{eq:F2a}
		\nabla_{ad^i_eg}ad^j_eg=0,\qquad
		\nabla^2_{ad^k_eg,ad^j_eg}e=0,
	\end{align}
which implies that (MF1)-(MF4) hold for \lms (in particular, (MF1) holds because \lms is assumed controllable). To prove necessity of (MF1)-(MF4), we will show that they are MF-invariant.
	All conditions (MF1)-(MF4) are expressed in a geometrical way, therefore they are invariant under diffeomorphisms. The conditions
	(MF1) and (MF2) are mechanical feedback invariant, see Remark \ref{rem:inv}. It remains to show that (MF3) and (MF4) are invariant under the  mechanical feedback $u=\gamma_{jk}(x)y^jy^k+\alpha(x)+\beta(x)\tilde{u}$.
	For the closed-loop system, {denoted by "$\sim$"}, the Christoffel symbols $\ti{\Gamma}^i_{jk}$ of $\ti{\nabla}$, $\ti{e}$, and $\ti{g}$ are, respectively, given by
	\begin{align}
		\label{eq:cl1}
		\tilde{\Gamma}^i_{jk}=\Gamma^i_{jk}- g^i\gamma_{jk},\qquad
		\tilde{e}=e+g\alpha,\qquad
		\tilde{g}=g\beta.
	\end{align}
	For any $X,Y\in \mathfrak{X}(Q)$, we have $\tilde{\nabla}_{X}Y=\nabla_{X}Y- \gamma(X,Y)g=\nabla_{X}Y \mod \mathcal{E}^0$, where $\gamma(X,Y)=\gamma_{jk}X^jY^k\in C^{\infty}(Q)$,
	therefore		
	\begin{align*}
		\tilde{\nabla}_{ad_{\tilde{e}}^i\tilde{g}}\tilde{g}=\nabla_{ad^i_{\tilde{e}}\tilde{g}}\tilde{g}-\gamma(ad_{\tilde{e}}^i\tilde{g},\tilde{g})g=\nabla_{ad^i_{\tilde{e}}\tilde{g}}\tilde{g} \mod \mathcal{E}^0.
	\end{align*}
	
	By $\nabla_X\tilde{g} =\nabla_X\left( g \beta\right) = \nabla_Xg +  \left(L_X\beta \right)g$, 
	it follows that instead of calculating $\nabla_{ad_{\tilde{e}}^i\tilde{g}}\tilde{g}$ it is enough to calculate $\nabla_{ad_{\tilde{e}}^i\tilde{g}}g$, since the second term $\left( L_X\beta \right)g\in \mathcal{E}^0$. For i=0, we have
	$	\nabla_{\tilde{g}}g=\nabla_{\left( g \beta\right) }g=\beta \nabla_{g}g\in \mathcal{E}^0$.
	It is easy to show that for any ${1}\leq j\leq n-1$, we have
	\begin{align}
		\label{eq:Ff}
		ad_{\tilde{e}}^j\tilde{g}=\beta ad_e^jg +d^{j-1},
	\end{align}
	where $d^{j-1}\in \mathcal{E}^{j-1}$.	 Assume $\nabla_{ad_{\tilde{e}}^l\tilde{g}}g\in \mathcal{E}^0$, for $0\leq l \leq i-1 $. Then, by formula \eqref{eq:Ff}, $\nabla_{ad_{\tilde{e}}^i\tilde{g}}g=\beta \nabla_{ad^i_eg}g+\nabla_{d^{i-1}}g\in \mathcal{E}^0$,
	because the first term is in $\mathcal{E}^0$ by (MF3) and the second by the induction assumption.	We have thus proved necessity of (MF3).
	
	To show necessity of (MF4), using Lemma \ref{prop:cov2}, calculate
	\begin{align}
		\label{eq:mc4n}
		\begin{split}
			&\tilde{\nabla}^2_{X,Y}Z=\tilde{\nabla}_X\tilde{\nabla}_YZ-\tilde{\nabla}_{\tilde{\nabla}_XY}Z\\&=\tilde{\nabla}_X\left(\nabla_YZ- \gamma(Y,Z) g\right)-\tilde{\nabla}_{\left(\nabla_{X}Y- \gamma(X,Y) g \right) }Z
		\\&	=\nabla^2_{X,Y}Z-\gamma(Y,Z) \nabla_Xg + \gamma(X,Y)\nabla_{g}Z \mod \mathcal{E}^0.
		\end{split}
	\end{align}
	By the above formula, we get
	\begin{align*}
		\begin{split}
			\ti{\nabla}^2_{ad^k_{\ti{e}}\ti{g},ad^j_{\ti{e}}\ti{g}}\ti{e}=&\nabla^2_{ad^k_{\ti{e}}\ti{g},ad^j_{\ti{e}}\ti{g}}\ti{e}-\gamma(ad^j_{\ti{e}}\ti{g},\ti{e}) \nabla_{ad^k_{\ti{e}}\ti{g}}g \\&+ \gamma(ad^k_{\ti{e}}\ti{g},ad^j_{\ti{e}}\ti{g})\nabla_{g}\ti{e} \mod \mathcal{E}^0.
		\end{split}
	\end{align*}
	The second term, on the right hand side, is in $\mathcal{E}^0$ (by (MF3) and its invariance), while the third term is a function multiplying 
	\begin{align*}
		\nabla_{g}\tilde{e}&=\nabla_{g}\left(e+ g\alpha \right) =
		    \nabla_{g}e +    \alpha\nabla_{g}g + L_{g}\alpha \ g   \in \mathcal{E}^1, 
	\end{align*}
	since for \lms we have $ \nabla_{g}e=-ad_eg=-Eb\in\mathcal{E}^1$.
	
	The first term $\nabla^2_{ad^k_{\ti{e}}\ti{g},ad^j_{\ti{e}}\ti{g}}\ti{e}$ is, by \eqref{eq:Ff} and Lemma \ref{prop:cov2}(i), a linear combination with smooth coefficients of $\nabla^2_{ad_e^ig,ad_e^lg}\ti{e}$, with $0\leq i \leq k$ and  $0\leq l \leq j$. Thus we calculate	
	$
	\nabla^2_{ad^i_eg,ad^l_eg}\tilde{e}= \nabla^2_{ad^i_eg,ad^l_eg}e +  \nabla^2_{ad^i_eg,ad^l_eg}( g\alpha).
	$	
	The first term vanishes since \eqref{eq:F2a} holds for \lms. We calculate the second term using Lemma \ref{prop:cov2}(iii), and we have $	
	\nabla^2_{ad^i_eg,ad^l_eg}( g\alpha)= \alpha \nabla^2_{ad^i_eg,ad^l_eg}g + L_{ad^i_eg}\alpha \nabla_{ad^l_eg}g + L_{ad^l_eg}\alpha \nabla_{ad^i_eg}g + (\nabla^2_{ad^i_eg,ad^l_eg}\alpha) g \in \mathcal{E}^0$ because the first three terms vanish, due to \eqref{eq:F2a}, and  the last one is in $\mathcal{E}^0$. Summarizing the above calculations, we conclude that $
	\ti{\nabla}^2_{ad^k_{\ti{e}}\ti{g},ad^j_{\ti{e}}\ti{g}}\ti{e} \in \mathcal{E}^1=\tilde{\mathcal{E}}^1$, 
	which proves necessity of (MF4).
	
	\textit{Sufficiency}. {We will transform the system \ms, satisfying (MF1)-(MF4), into \lms in two steps. In the first step, we will normalize the vector fields $e$ and $g$ and show that condition (MF4) implies zeroing some of the Christoffel symbols $\Gamma^i_{jk}$, which exhibit a triangular form in the normalizing coordinates. In the second step, we compensate the remaining Christoffel symbols.}
	
	By conditions (MF1)-(MF2), there exists a function $h$ satisfying $L_{ad_e^jg}h=0$, for $0\leq j \leq n-2$, and $L_{ad_e^{n-1}g}h\neq0$, and thus $(\ti{x},\ti{y})=(\phi(x),\fp{\phi}{x}(x)y)$ is a local mechanical diffeomorphism, where $\phi(x)=(L_e^{n-1}h,\ldots,L_eh, h)^T$ that can be completed by a feedback transformation $(\alpha,\beta{,0})$ that map, respectively, $\beta g$ into $\tilde{g}=(1,0,\ldots,0)^T$, $e+g\alpha$ into $\tilde{e}=(0,\ti{x}^1,\ldots,\ti{x}^{n-1})^T$, {and $\Gamma^i_{jk}$ into $\ti \Gamma^i_{jk}$}, see {the} classical results of feedback linearization \cite{JR}, \cite{nij}, \cite{Isidori}. Then, $(\Phi,\alpha,\beta,\gamma)\in {G_{MF}}$, where $(\ti{x},\ti{y})=\Phi(x,y)=\left(\phi(x), \fp{\phi}{x}(x)y \right) $ with $\phi,\alpha,\beta$ just defined and $\gamma_{jk}=\tilde{\Gamma}^1_{jk}(\ti{x})$, brings \eqref{eq:ms} into (we drop "tildas" for readability)
	\begin{align}
		\label{eq:nf}
		\begin{split}
			\dot{x}^{1}&=y^{1}\\
			\dot{x}^{i}&=y^{i}
		\end{split}
		\begin{split}
			\dot{y}^{1}&=u\\
			\dot{y}^{i}&=-\Gamma^{i}_{jk}y^jy^k +x^{i-1}, \quad 2 \leq i \leq n,
		\end{split}		
	\end{align}
{to which Lemma \ref{lem:formula} applies.}

{We will show that the Christoffel symbols $\Gamma^i_{jk}$ of \eqref{eq:nf} satisfy 
	\begin{align}
		\label{eq:nft2}
		\begin{split}
			&{\Gamma}^i_{kj}=0\qquad \text{for $1\leq k \leq n-1$, $1\leq j \leq i \leq n$},\\
			&{\Gamma}^i_{nj} =  \begin{cases}
				\;	0  &  \qquad\text{for } 1\leq j < i \leq n\\
				\;	\lambda(x^n) &  \qquad\text{for } 2\leq j=i \leq n.
			\end{cases}
		\end{split}
	\end{align}}

	For {system \eqref{eq:nf}}, we have $ad_e^{k-1}g=(-1)^{k-1}\fp{}{x^k}$ and, in particular, $g=\fp{}{x^1}$.
	 {Calculate} 
	$\nabla_{ad_e^{k-1}g}g=(-1)^{k-1}\nabla_{\fp{}{x^k}}g^i\fp{}{x^i}=(-1)^{k-1}\nabla_{\fp{}{x^k}}\fp{}{x^1}=(-1)^{k-1}\Gamma^i_{k1}\fp{}{x^i}$.
	It follows that $\Gamma^i_{k1}=\Gamma^i_{1k}=0$, for $2\leq i\leq n$ by (MF3), and for $i=1$ by the above form.
	
	Rewrite (MF4) as $\nabla^2_{ad_e^{k-1}g,ad_e^{j-1}g}e=0 \mod \mathcal{E}^1$, for $1\leq j,k\leq n$, and apply it successively for $j=1,\ldots,n$ and for all $1\leq k \leq n$. For $j=1$, first calculate 
	\begin{align*}
		&\nabla_ge=\nabla_{\fp{}{x^1}}e=\fp{}{x^2}+\Gamma^i_{1s}e^s\fp{}{x^i}=\fp{}{x^2} \text{ \quad and then }\\
		&\nabla_{ad_e^{k-1}g}\left(\nabla_ge \right) = (-1)^{k-1}\nabla_{\fp{}{x^k}}\fp{}{x^2}= (-1)^{k-1}\Gamma^i_{k2}\fp{}{x^i}.
	\end{align*}
	On the other hand, $	\nabla_{ad_e^{k-1}g}g=(-1)^{k-1}\nabla_{\fp{}{x^k}}\fp{}{x^1}= (-1)^{k-1}\Gamma^1_{k1}\fp{}{x^1}=0$ and hence $\nabla_{\nabla_{ad_e^{k-1}g}g}e=0$. Thus, by \eqref{eq:cd2},
	\begin{align*}
		\nabla^2_{ad_e^{k-1}g,g}e=\nabla_{ad_e^{k-1}g}\left(\nabla_ge \right)- \nabla_{\nabla_{ad_e^{k-1}g}g}e=\\= (-1)^{k-1}\Gamma^i_{k2}\fp{}{x^i}=0 \mod \mathcal{E}^1,
	\end{align*}
	implying that $\Gamma^i_{k2}=\Gamma^i_{2k}=0$ for any $3\leq i\leq n$.

	For $j=2$, calculate
	\begin{align*}
		\nabla_{ad_eg}e=-\nabla_{\fp{}{x^2}}e=-\fp{}{x^3}+\Gamma^i_{2s}e^s\fp{}{x^i}=-\fp{}{x^3} -d
	\end{align*}
	where $d=d^1(x)\fp{}{x^1}+d^2(x)\fp{}{x^2}\in\mathcal{E}^1$, and then 
	\begin{align*}
		&\nabla_{ad_e^{k-1}g}\left(\nabla_{ad_eg}e \right) = (-1)^{k}\nabla_{\fp{}{x^k}}\left(\fp{}{x^3} +d \right)= \\&= (-1)^{k}\left( \Gamma^i_{k3}+\Gamma^i_{k1}d^1+\Gamma^i_{k2}d^2\right) \fp{}{x^i}=\\&= (-1)^{k} \Gamma^i_{k3}\fp{}{x^i} \mod \mathcal{E}^1.
	\end{align*}
	On the other hand, 
	\begin{align*}
		\nabla_{ad_e^{k-1}g}ad_eg&=(-1)^{k}\nabla_{\fp{}{x^k}}\fp{}{x^2}= (-1)^{k}\Gamma^i_{k2}\fp{}{x^i}=\\&= (-1)^{k}\left( \Gamma^1_{k2}\fp{}{x^1}+\Gamma^2_{k2}\fp{}{x^2}\right)
	\end{align*} and $
	\nabla_{\nabla_{ad_e^{k-1}g}ad_eg}e=(-1)^{k}\Gamma^2_{k2}\fp{}{x^3}\mod \mathcal{E}^1$.
	It follows that, modulo $\mathcal{E}^1$,
	\begin{align*}
		\nabla^2&_{ad_e^{k-1}g,ad_eg}e= (-1)^{k}\left(\sum_{i=4}^{n}\Gamma^i_{k3}\fp{}{x^i} + (\Gamma^3_{k3}-\Gamma^2_{k2})\fp{}{x^3} \right),
	\end{align*}
	and, using (MF4), we conclude $\Gamma^i_{k3}=\Gamma^i_{3k}=0$ for any $4\leq i\leq n$ and $\Gamma^3_{k3}=\Gamma^2_{k2}$.
	
	{	Following the same line (with a more tedious calculation), one can prove the general induction step. Namely, assuming, for a fixed $j$,
		\begin{align}
			\label{eq:F2b}
			\begin{split}
				&\Gamma^{j}_{kj}=\Gamma^{j-1}_{kj-1}\\
				&\Gamma^i_{ks}=\Gamma^i_{sk}=0 \quad s+1 \leq i\leq n, \quad 1\leq s\leq j,
			\end{split}
		\end{align}
		one shows by calculating $\nabla^2_{ad_e^{k-1}g, ad_e^{j-1}g}e$,  {with the help of \eqref{eq:formula} of Lemma \ref{lem:formula},} that
		\begin{align*}
			&\Gamma^{j+1}_{kj+1}=\Gamma^{j}_{kj} && \\
			&\Gamma^i_{kj+1}=0 &&\text{ for }j+2 \leq i\leq n 
		\end{align*}
		and thus, by the induction assumption and symmetry of the Christoffel symbols, 
		\begin{align}
				\label{eq:iarg}
			\Gamma^i_{ks}=\Gamma^i_{sk}=0 \quad s+1\leq  i\leq n, \quad 1\leq s\leq j+1.
		\end{align}
		It follows that for each $1\leq k\leq n$ the matrices {consisting} of Christoffel symbols $(\Gamma^i_{kj})$, for $2\leq i,j \leq n$ are upper triangular.
		By the induction argument, \eqref{eq:F2b} holds for all $2\leq j\leq n$ and implies, for any $1\leq k \leq n-1$, 
		\begin{align*}
			\Gamma^{2}_{k2}=\ldots=\Gamma^{n-1}_{kn-1}=\Gamma^n_{kn}=0.
		\end{align*}
		since $\Gamma^n_{kn}=\Gamma^n_{nk}=0$ (as $n>k$). On the other hand, for $k=n$,  \eqref{eq:F2b} implies 
		\begin{align*}
			\Gamma^{2}_{n2}=\ldots=\Gamma^{n-1}_{nn-1}=\Gamma^n_{nn}=\lambda(x)
		\end{align*}
		for a function $\lambda(x)$.}
	Therefore for each $1\leq k\leq n$ the matrices $(\Gamma^i_{kj})$, for $2\leq i,j \leq n$, are strictly upper triangular, and the last one, for   $k=n$, is upper triangular with all diagonal elements equal to each other, which we denote by $\lambda(x)$. The matrices read
	\begin{align*}
		\left(\Gamma^i_{kj} \right)=\left(\begin{array}{ccccccc}
			0&  \Gamma^2_{k3}& \Gamma^2_{k4}  & \ldots & \Gamma^2_{kn-2} & \Gamma^2_{kn-1} & \Gamma^2_{kn} \\ 
			0&  0&  \Gamma^3_{k4}& \ldots  & \Gamma^3_{kn-2} & \Gamma^3_{kn-1} & \Gamma^3_{kn} \\ 
			&  &  \ddots&&    & &   \\ 
			0&  0&  0& \ldots &0  &\Gamma^{n-2}_{kn-1}  &\Gamma^{n-2}_{kn}  \\ 
			0&  0&  0& \ldots &0  &0  &\Gamma^{n-1}_{kn}  \\ 
			0&  0&  0& \ldots &0  &0  &0 
		\end{array}  	 \right),
	\end{align*}
	for $1\leq k \leq n-1$, and
	\begin{align*}
		\left(\Gamma^i_{nj} \right)=\left(\begin{array}{ccccccc}
			\lambda &  \Gamma^2_{n3}& \Gamma^2_{n4}  & \ldots & \Gamma^2_{nn-2} & \Gamma^2_{nn-1} & \Gamma^2_{nn} \\ 
			0&  \lambda &  \Gamma^3_{n4}& \ldots  & \Gamma^3_{nn-2} & \Gamma^3_{nn-1} & \Gamma^3_{nn} \\ 
			&  &  \ddots&&    & &   \\ 
			0&  0&  0& \ldots &\lambda   &\Gamma^{n-2}_{nn-1}  &\Gamma^{n-2}_{nn}  \\ 
			0&  0&  0& \ldots &0  &\lambda   &\Gamma^{n-1}_{nn}  \\ 
			0&  0&  0& \ldots &0  &0  &\lambda 
		\end{array}  	 \right),
	\end{align*}
{and are thus of the desired triangular structure  \eqref{eq:nft2} and it remains to prove that $\lambda=\lambda(x^n)$.}
	Note that in the above matrices we skip the first row $\Gamma^1_{kj}$ and the first column $\Gamma^i_{k1}$. This is due to the fact that $\Gamma^1_{{kj}}=0$ (which can always be achieved by a suitable feedback transformation) and $\Gamma^i_{k1}=0$ by (MF3).
	{
Notice that we have $\mathcal{E}^{n-2}=\spn{\fp{}{x^1},\ldots,\fp{}{x^{n-1}}}$ and thus applying \eqref{eq:formula} of Lemma~\ref{lem:formula}, for $j=n$ and any $1 \leq k \leq n$, we conclude {(set $\Gamma^n_{kn+1}=0$)}
\begin{align}
	\begin{split}
		&(-1)^{n+k-2}	\nabla^2_{ad_e^{k-1}g,ad_e^{n-1}g}e=\nabla^2_{\fp{}{x^k},\fp{}{x^n}}e\\
		& = \bigg( \fp{{\Gamma}^n_{n s}}{{x}^k} {e}^s + {\Gamma}^n_{nk+1} +{\Gamma}^n_{kn+1} - {\Gamma}^{n-1}_{kn}\\
		&+({\Gamma}^d_{ns}{\Gamma}^n_{kd}- {\Gamma}^d_{kn}{\Gamma}^n_{ds}){e}^s\bigg) \!\fp{}{{x}^n} \mod  \mathcal{E}^{n-2}\\
		& = \bigg( \fp{\lambda}{{x}^k} {e}^n + {\Gamma}^n_{nk+1} - {\Gamma}^{n-1}_{kn}\bigg) \fp{}{{x}^n} \mod \mathcal{E}^{n-2},
	\end{split}
	\label{eq:kn}
\end{align} 
{since, }{due to the triangular structure \eqref{eq:iarg}}, $\Gamma^n_{n s}=0$ except for $s=n$ giving $\Gamma^n_{n n}= \lambda$ and, {moreover, the equality} ${\Gamma}^d_{ns}{\Gamma}^n_{kd}\!-\!{\Gamma}^d_{kn}{\Gamma}^n_{ds}=0$ {holds}. Indeed, {in the latter,} $\Gamma^n_{kd}=0$ except $d=k=n$ giving ${\Gamma}^n_{ns}{\Gamma}^n_{nn}- {\Gamma}^n_{nn}{\Gamma}^n_{ns}=0$ and $\Gamma^n_{ds}=0$ except for $d=s=n$ giving ${\Gamma}^n_{nn}{\Gamma}^n_{kn}- {\Gamma}^n_{kn}{\Gamma}^n_{nn}=0$.}

{For \eqref{eq:kn} we will apply (MF4) in three cases. First, if $1\leq k\leq n-2$, then, modulo $\mathcal{E}^{n-2}$, we have
	\begin{align*}
		 \left(\fp{\lambda}{x^k}e^n+\Gamma^n_{nk+1} -\Gamma^{n-1}_{kn} \right) \fp{}{x^n}= \left( \fp{\lambda}{x^k}x^{n-1}\right) \fp{}{x^n}=0,
	\end{align*}
	since all $\Gamma^n_{nk+1}=0$ and all $\Gamma^{n-1}_{kn}=0$ by \eqref{eq:iarg} and $k\leq n-2$.
	Second, for $k=n-1$, we have modulo $\mathcal{E}^{n-2}$,
	\begin{align*}
		& \left(\fp{\lambda}{x^{n-1}}e^n+\Gamma^n_{nn} -\Gamma^{n-1}_{n-1n} \right) \!\fp{}{x^{n}}= \left(\fp{\lambda}{x^{n-1}}e^n+\lambda -\lambda \right) \fp{}{x^{n}}\\&= \left( \fp{\lambda}{x^{n-1}}x^{n-1}\right) \fp{}{x^n}=0.
	\end{align*}
	{Therefore $\fp{\lambda}{x^k}=0$, for $1\leq k\leq n-1$,} implying that $\lambda$ is a function of the last variable $x^n$ only, i.e.  $\lambda=\lambda(x^n)$, which gives the system in the desired form \eqref{eq:nft2}
	Third, for $k=n$, we have modulo $\mathcal{E}^{n-2}$,
	\begin{align*}
		\bigg( \fp{\lambda}{{x}^n} {e}^n\! +\! {\Gamma}^n_{nn+1}\! -\! {\Gamma}^{n-1}_{nn}\bigg) \!\fp{}{{x}^n}\!=\!\bigg( \fp{\lambda}{{x}^n} {x}^{n-1}\! - {\Gamma}^{n-1}_{nn}\bigg)\! \fp{}{{x}^n}\!=0,
	\end{align*}}
	implying that ${\Gamma^{n-1}_{nn}=L_e\lambda}$, since $\fp{\lambda (x^n)}{x^n}x^{n-1}=L_e\lambda$.
	
	Now, transform system \eqref{eq:nf}, {satisfying \eqref{eq:nft2},} via the local mechanical diffeomorphism $\Phi: \T Q\rightarrow \T \bar{Q}$ 
	\begin{align}
		\label{eq:difha}
		\begin{split}
			&\bar{x}=\phi(x)\\
			&\bar{y}=D\phi(x)y,
		\end{split}
		\text{\quad where }
		\begin{split}
			&\phi(x)=\left(
				L^{n-1}_eh,
				\ldots,
				L_eh,
				h
		  \right)^T,
		\end{split}
	\end{align}
with $h(x^n)=\int^{x^n}_0 \Lambda(s_2) ds_2$, where $\Lambda(s_2)=\exp\left(\int^{s_2}_0 \lambda(s_1) ds_1 \right)$.

	Denote by $\bar{\Gamma}^i_{jk}, \bar{e}, \bar{g}$ the objects of the system expressed in coordinates $\bar{x}=\phi(x)$. Applying  feedback $\bar{u}=-\bar{\Gamma}^1_{jk}\bar{y}^j\bar{y}^k+L^n_eh+uL_gL_e^{n-1}h$, the transformed system becomes
	\begin{align}
		\label{eq:u2}
		\begin{split}
			\dot{\bar{x}}^{1}&=\bar{y}^{1}\\
			\dot{\bar{x}}^{i}&=\bar{y}^{i}
			\end{split}
			\begin{split}
				\dot{\bar{y}}^{1}&=\bar{u}\\
				\dot{\bar{y}}^{i}&=-\bar{\Gamma}^{i}_{jk}\bar{y}^j\bar{y}^k +\bar{x}^{i-1}, \quad 2 \leq i \leq n,
			\end{split}
	\end{align}
 whose vector fields are {
$\bar{e}=\bar{x}^{i-1}\fp{}{\bar{x}^i}${, where $x^0=0$,} and $	{\bar{g}=\fp{}{\bar{x}^1}}.
$}
{Transformed system \eqref{eq:u2} is still of the form \eqref{eq:nf} and at the moment we ignore how $\Gamma^i_{jk}$ have been changed into $\bar{\Gamma}^i_{jk}$. Below we will prove that all $\bar{\Gamma}^i_{jk}$ vanish. To this end, we first} 
	calculate explicitly the time-evolution of the pair $\left(\bar{x}^n,\bar{y}^n \right) $
	\begin{align*}
		\dot{\bar{x}}^n&=\frac{d}{dt}h(x^n)=\Lambda(x^n) \dot{x}^n=\Lambda(x^n) y^n=\bar{y}^n\\
		\dot{\bar{y}}^n&=\frac{d}{dt}\left(\Lambda(x^n) y^n \right)= \Lambda(x^n) \lambda(x^n) \dot{x}^ny^n + \Lambda(x^n) \dot{y}^n \\&=\Lambda(x^n) \lambda(x^n) y^ny^n + \Lambda(x^n) \dot{y}^n\\
		&=\Lambda(x^n) \lambda(x^n) y^ny^n+ \Lambda(x^n)\left( -\Gamma^n_{nn}(x^n) y^n y^n + x^{n-1}\right)\\&=  \Lambda(x^n)  x^{n-1}=\bar{x}^{n-1},
	\end{align*}
	since $\bar{x}^{n-1}=L_eh=\Lambda(x^n) x^{n-1}$. It follows that $\bar{\Gamma}^n_{jk}=0$, for all $1\leq k,j \leq n$.

{	For transformed system \eqref{eq:u2}, we rewrite \eqref{eq:formula} by adding "bars" as}
\begin{align}
			\label{eq:F4}
	\begin{split}
		&\nabla^2_{ad_{\bar{e}}^{k-1}\bar{g}, ad_{\bar{e}}^{j-1}\bar{g}}\bar{e}=(-1)^{j+k}\bigg( \fp{\bar{\Gamma}^i_{j s}}{\bar{x}^k} \bar{e}^s + \bar{\Gamma}^i_{jk+1}\\& +\bar{\Gamma}^i_{kj+1}  +(\bar{\Gamma}^d_{js}\bar{\Gamma}^i_{kd}- \bar{\Gamma}^d_{kj}\bar{\Gamma}^i_{ds})\bar{e}^s 
		- \bar{\Gamma}^{i-1}_{kj}\bigg) \fp{}{\bar{x}^i}
	\end{split}
\end{align}
and by (MF4), we have
	\begin{align*}
		\nabla^2_{ad_{\bar{e}}^{k-1}\bar{g}, ad_{\bar{e}}^{j-1}\bar{g}}\bar{e}=(-1)^{j+k}\bar{a}_{kj}^n(\bar{x}) \fp{}{\bar{x}^n}=0 \mod \mathcal{E}^{n-2},
	\end{align*}
	where $\bar{a}_{kj}^n(\bar{x})= \fp{\bar{\Gamma}^n_{j s}}{\bar{x}^k} \bar{e}^s + \bar{\Gamma}^n_{jk+1} +\bar{\Gamma}^n_{kj+1}  + (\bar{\Gamma}^d_{js}\bar{\Gamma}^n_{kd}- \bar{\Gamma}^d_{kj}\bar{\Gamma}^n_{ds})\bar{e}^s - \bar{\Gamma}^{n-1}_{kj}$, which implies (since $\bar{\Gamma}^n_{kj}=0$, for $1\leq j,k \leq n$) that 
	$
		\bar{a}_{kj}^n(\bar{x})= \bar{\Gamma}^{n-1}_{kj}=0
	$.
	Now assume $\bar{\Gamma}^i_{kj}=0$ for a certain $1\leq i \leq n-1$  and any $1\leq j,k \leq n$. Then \eqref{eq:F4} and (MF4) imply $\bar{\Gamma}^{i-1}_{kj}=0$. Therefore we have proved that all Christoffel symbols of \eqref{eq:u2} vanish and thus the system is  a linear controllable \lms, since the vector field $\bar{e}=\bar{x}^{i-1}\fp{}{\bar{x}^i}$ is linear and $\bar{g}=\fp{}{\bar{x}^1}$ is constant.  		
\end{proof}

The above theorem does not work for systems with 2 degrees of freedom, i.e. for n=2, as that case is too restrictive for involutivity, see Remark \ref{rem:2d} below. Therefore we state the following theorem for MF-linearization of \ms with 2 degrees of freedom.

\begin{theorem}
	\label{thm:MFn2}
	A mechanical system \eqref{eq:ms} with 2 degrees of freedom is, locally around $x_0$, MF-linearizable to a controllable linear $(\mathcal{LMS})$ if and only if it satisfies in a neighborhood of $x_0$
	\begin{enumerate}[(MF1)']
		\item $g$ and $ ad_eg$ are independent at $x_0$,
		\setcounter{enumi}{2}
		\item $\nabla_{g}\, g\in \mathcal{E}^0$ and $\nabla_{ad_eg}\, g\in \mathcal{E}^0$,
		\setcounter{enumi}{4}
		\item $\nabla^2_{g,ad_eg} \, ad_eg-\nabla^2_{ad_eg,g} \, ad_eg\in \mathcal{E}^0$.
	\end{enumerate}
\end{theorem}
\begin{remark}
	\label{rem:2d}
	If $n=2$, then $\mathcal{E}^0$ is of rank 1, thus involutive and (MF2) is trivially satisfied, and so is (MF4) because $\mathcal{E}^1=\T Q$ (cf. Theorem \ref{thm:MFC}). Therefore (MF2)' and (MF4)' are absent and replaced by (MF5)' that guarantees that we can compensate the Christoffel symbols (as do (MF3)-(MF4) for $n\geq 3$).
\end{remark}
\begin{proof}
	\textit{Necessity.}
	Note that (MF1)' is equivalent to (MF1) and (MF3)' is (MF3) of Theorem \ref{thm:MFC}. {Although Theorem \ref{thm:MFC} applies to $n\geq 3$, the necessity part of its proof remains valid for any $n\geq 2$ so it shows necessity of (MF1)'-(MF3)'.}
	Therefore we need to show necessity of (MF5)'. For a controllable \lms we have $\Gamma^i_{jk}=0$, $g=b$ and $ad_eg=-Eb$ are independent, and
	\begin{align}
		\label{eq:lms2}
		\begin{split}
				\nabla_{ad^i_eg}ad^j_eg=0,\
		\nabla^2_{ad^j_eg,ad^k_eg} ad^i_eg=0,\
		\left[ ad^j_eg,ad^k_eg\right] =0,
		\end{split}
	\end{align}
	for $0\leq i,j,k \leq 1$.
	We will use formula \eqref{eq:mc4n} to show that (MF5)' is invariant under mechanical feedback.
	Denote $\ti{\nabla}, \ti{e}, \ti{g}, \gamma$ as in \eqref{eq:cl1}. Then we calculate
	\begin{align*}
		\begin{split}
			\ti{\nabla}^2_{\ti{g},ad_{\ti{e}}\ti{g}}ad_{\ti{e}}\ti{g}=&\nabla^2_{\ti{g},ad_{\ti{e}}\ti{g}}ad_{\ti{e}}\ti{g}-\gamma(ad_{\ti{e}}\ti{g},ad_{\ti{e}}\ti{g}) \nabla_{\ti{g}}\ti{g}\\ &+\gamma(\ti{g},ad_{\ti{e}}\ti{g})\nabla_{\ti{g}}ad_{\ti{e}}\ti{g} \mod \mathcal{E}^0,\\
			\ti{\nabla}^2_{ad_{\ti{e}}\ti{g},\ti{g}}ad_{\ti{e}}\ti{g}=&\nabla^2_{ad_{\ti{e}}\ti{g},\ti{g}}ad_{\ti{e}}\ti{g}-\gamma(g,ad_{\ti{e}}\ti{g}) \nabla_{ad_{\ti{e}}\ti{g}}\ti{g}\\ &+ \gamma(ad_{\ti{e}}\ti{g},\ti{g})\nabla_{\ti{g}}ad_{\ti{e}}\ti{g} \mod \mathcal{E}^0.
		\end{split}
	\end{align*}
	The second terms of the right hand side of both equations are in $\mathcal{E}^0$ due to the feedback invariance of (MF3)', while the third terms are equal since $\gamma(X,Y)=\gamma(Y,X)$ is symmetric. Therefore we conclude 
	\begin{align*}
	&	\ti{\nabla}^2_{\ti{g},ad_{\ti{e}}\ti{g}}ad_{\ti{e}}\ti{g}-\ti{\nabla}^2_{ad_{\ti{e}}\ti{g},\ti{g}}ad_{\ti{e}}\ti{g}\\&=	\nabla^2_{\ti{g},ad_{\ti{e}}\ti{g}}ad_{\ti{e}}\ti{g}-\nabla^2_{ad_{\ti{e}}\ti{g},\ti{g}}ad_{\ti{e}}\ti{g}\mod \mathcal{E}^0.
	\end{align*}
	Denoting $ad_{\ti{e}}\ti{g}=\beta ad_eg+d^0 g$ (see \eqref{eq:Ff}) and by Lemma \ref{prop:cov2}\! (i), we~have 
	\begin{align*}
		\nabla^2_{\ti{g},ad_{\ti{e}}\ti{g}}ad_{\ti{e}}\ti{g}&= \nabla^2_{\beta g,\beta ad_eg+d^0 g  }ad_{\ti{e}}\ti{g}\\ \qquad&=\beta^2 \nabla^2_{g,ad_eg}ad_{\ti{e}}\ti{g}+ \beta d^0 \nabla^2_{g,g}ad_{\ti{e}}\ti{g}\\
		\nabla^2_{ad_{\ti{e}}\ti{g},\ti{g}}ad_{\ti{e}}\ti{g}&= \nabla^2_{\beta ad_eg+d^0 g , \beta g }ad_{\ti{e}}\ti{g}\\ \qquad&=\beta^2 \nabla^2_{ad_eg,g}ad_{\ti{e}}\ti{g}+ \beta d^0 \nabla^2_{g,g}ad_{\ti{e}}\ti{g},
	\end{align*}
	where the last terms on the right are equal, implying
	\begin{align*}
	&	\nabla^2_{\ti{g},ad_{\ti{e}}\ti{g}}ad_{\ti{e}}\ti{g}-	\nabla^2_{ad_{\ti{e}}\ti{g},\ti{g}}ad_{\ti{e}}\ti{g}\\ 
	&\quad =\beta^2\left(\nabla^2_{g,ad_eg}ad_{\ti{e}}\ti{g}-\beta^2 \nabla^2_{ad_eg,g}ad_{\ti{e}}\ti{g} \right) 
	\end{align*}
	and it remains to prove that
	$
	\nabla^2_{g,ad_eg}ad_{\ti{e}}\ti{g}-\nabla^2_{ad_eg,g}ad_{\ti{e}}\ti{g}\in \mathcal{E}^0,
	$
	which we show using Lemma \ref{prop:cov2}(iii), where $X,Y$ stand for either $g$ or $ad_eg$. Denote $\nabla_X\beta=L_X\beta$ and $\nabla^2_{X,Y}\beta=L_XL_Y\beta-L_{\nabla_XY}\beta$ (see Lemma \ref{prop:cov2}) and calculate
	\begin{align*}
		&\nabla^2_{X,Y}ad_{\ti{e}}\ti{g}=\nabla^2_{X,Y}\left(\beta ad_eg+d^0 g \right) = \beta \nabla^2_{X,Y}ad_eg\\
		&\quad + L_X\beta \nabla_Yad_eg + L_Y\beta \nabla_X ad_eg+ \left( \nabla^2_{X,Y}\beta\right)  ad_eg\\ 
		&\quad +d^0 \nabla^2_{X,Y} g + L_Xd^0 \nabla_Y g +L_Yd^0 \nabla_X g + \left( \nabla^2_{X,Y}d^0\right)  g\\
		&= \left( \nabla^2_{X,Y}\beta \right) ad_eg \mod \mathcal{E}^0,
	\end{align*}
	since all $\nabla^2_{X,Y}X=0$ and $\nabla_X Y=0$ , see \eqref{eq:lms2}. Therefore we have
	\begin{align*}
	&	\nabla^2_{g,ad_eg}ad_{\ti{e}}\ti{g}-\nabla^2_{ad_eg,g}ad_{\ti{e}}\ti{g}\\ &\qquad=\left( \nabla^2_{g,ad_eg}\beta -  \nabla^2_{ad_eg,g}\beta\right)  ad_eg \mod \mathcal{E}^0.
	\end{align*}
	Finally, we calculate 
	\begin{align*}
		&\nabla^2_{g,ad_eg}\beta -  \nabla^2_{ad_eg,g}\beta=L_gL_{ad_eg}\beta-L_{\nabla_gad_eg}\beta\\
		&- \left(L_{ad_eg}L_{g}\beta-L_{\nabla_{ad_eg}g}\beta \right)
		=L_{\left[ g,ad_eg\right] }\beta=0 ,
	\end{align*}
	which shows necessity of (MF5)'.
	
	\textit{Sufficiency.}
	By (MF1)', $\rk \mathcal{E}^1=2$, and $\mathcal{E}^{0}=\spn{g}$ is of constant rank $1$ and thus always involutive, hence the system is, locally around $x_0$ ({since} $g(x_0)\neq0$), MF-equivalent to {(cf. \eqref{eq:nf})}
	\begin{align*}
		\begin{split}
			\dot{x}^{1}&=y^{1}\\
			\dot{x}^{2}&=y^{2}
		\end{split}
	\begin{split}
		\dot{y}^{1}&=u\\
		\dot{y}^{2}&=-\Gamma^{2}_{jk}y^jy^k +x^{2}.
	\end{split}
	\end{align*}
	We have $g=\fp{}{x^1}$, $ad_eg=-\fp{}{x^2}$ and now we calculate
	\begin{align*}
		\nabla_gg=\Gamma^2_{11}\fp{}{x^2}\qquad
		\nabla_{ad_eg}g=-\Gamma^2_{12}\fp{}{x^2},
	\end{align*} 
	which by (MF3)' are in $\mathcal{E}^0=\spn{\fp{}{x^1}}$, implying $\Gamma^2_{11}=\Gamma^2_{12}=\Gamma^2_{21}=0$.
	It follows
	$
	\nabla_gg=\nabla_{ad_eg}g=\nabla_{g}ad_eg=0$, and $
	\nabla_{ad_eg}ad_eg=\Gamma^2_{22}\fp{}{x^2}
	$
	and thus 
	\begin{align*}
		&\nabla^2_{g,ad_eg} \, ad_eg-\nabla^2_{ad_eg,g} \, ad_eg=\nabla_g \nabla_{ad_eg}ad_eg \\
		&- \nabla_{\nabla_gad_eg}ad_eg-\nabla_{ad_eg} \nabla_{g}ad_eg - \nabla_{\nabla_{ad_eg}g}ad_eg\\&=\nabla_g \nabla_{ad_eg}ad_eg=\nabla_{\fp{}{x^1}}\Gamma^2_{22}\fp{}{x^2}=\fp{\Gamma^2_{22}}{x^1}\fp{}{x^2}
	\end{align*}
	implying, by (MF5)', $\fp{\Gamma^2_{22}}{x^1}=0$, i.e. $\Gamma^2_{22}(x^2)=\lambda(x^2)$.
	
	Now, we transform the system via the local mechanical diffeomorphism $\Phi: \T Q\rightarrow \T \bar{Q}$ 
	(compare {to} \eqref{eq:difha})
	\begin{align*}
		\begin{split}
			&\bar{x}=\phi(x)\\
			&\bar{y}=D\phi(x)y,
		\end{split}
		\text{\qquad \qquad where }
		\begin{split}
			&\phi(x)=\left(	L_eh,h \right)^T ,
		\end{split}
	\end{align*}
	with
	${h(x^2)=\int^{x^2}_0 \Lambda(s_2) ds_2}$ and ${
	\Lambda(s_2)=\exp\left(\int^{s_2}_0 \lambda(s_1) ds_1 \right)}$.
	
	We calculate the evolution of the pair $(\bar x(t),\bar y(t))$ of transformed coordinates, using $\frac{d}{dt}h\left( x^2(t)\right) =\Lambda\left( x^2(t)\right) \dot{x}^2(t) $ and  $\frac{d}{dt}\Lambda\left( x^2(t)\right)= \lambda\left( x^2(t)\right) \Lambda\left( x^2(t)\right) \dot{x}^2(t)$; \ first we get
	\begin{align*}
		\dot{\bar{x}}^2&=\frac{d}{dt}h(x^2)=\Lambda(x^2) y^2=\bar{y}^2\\
		\dot{\bar{y}}^2&=\Lambda(x^2) \lambda(x^2) y^2y^2 + \Lambda(x^2) \dot{y}^2=\Lambda(x^2) \lambda(x^2) y^2y^2\\
		&+ \Lambda(x^2)\left( -\lambda(x^2) y^2 y^2 + x^{2}\right)=  \Lambda(x^2)  x^{1}=\bar{x}^{1}\\
	\text{and then}\quad	\dot{\bar{x}}^1&=\Lambda(x^2){y}^{1} + \frac{d}{dt}\Lambda\left( x^2(t)\right) x^1 y^2 =\bar{y}^1\\
		\dot{\bar{y}}^1&=-\bar{\Gamma}^1_{jk}\bar{y}^j\bar{y}^k+L_e^2h+uL_gL_eh,
	\end{align*}
	where we denote by $\bar{\Gamma}^1_{jk}$ the Christoffel symbols in the  $\dot{\bar{y}}^1$-equation of the transformed system. Applying  the feedback $\bar{u}=-\bar{\Gamma}^1_{jk}\bar{y}^j\bar{y}^k+L^2_eh+uL_gL_eh$, we get a controllable linear mechanical system in the canonical form
	$\dot{\bar{x}}^{1}=\bar{y}^{1},\  \dot{\bar{y}}^{1}=\bar{u},\ \dot{\bar{x}}^{2}=\bar{y}^{2},\  \dot{\bar{y}}^{2}=\bar{x}^1.$
\end{proof}
\section{Examples}
\label{sec:app}
\textbf{Example 1 (cont.):}
{For system \eqref{eq:ex1}, we have $g=\fp{}{x^2}$ and $ad_eg=-\fp{}{x^1}$ are independent. We check MF-linearizability using Theorem~\ref{thm:MFn2}. A simple calculation shows that $\nabla_gg=\nabla_{ad_eg}g=0\in\mathcal{E}^0$, but $\nabla^2_{g,ad_eg} \, ad_eg-\nabla^2_{ad_eg,g} \, ad_eg=\fp{}{x^1}\notin \mathcal{E}^0$, therefore the system is not MF-linearizable.}

{
Thus \eqref{eq:ex1} is an example of a system that is F-linearizable but not MF-linearizable. For such systems the choice is: either to F\text{-}linearize for the price of loosing the mechanical structure or to keep the mechanical structure but to get rid of the linearization.}

\textbf{Example 2:} Consider the equation of dynamics of the Inertia Wheel Pendulum \cite{iwp} {with constant parameters $m_0,m_d,J_2$}:
\begin{align*}
		\dot{x}^1=y^1,\quad
		\dot{x}^2=y^2,\quad
		\dot{y}^1=e^1+g^1 u,\quad
		\dot{y}^2=e^2+g^2 u,
\end{align*}
 $	e^1=\frac{m_0}{m_d}\sin x^1$, $e^2=-\frac{m_0}{m_d}\sin x^1$,  $g^1=-\frac{1}{m_d}$,  $g^2=\frac{m_d+J_2}{J_2 m_d}$.

We will verify whether the conditions of Theorem \ref{thm:MFn2} are satisfied. First, we calculate $ad_eg=(	\frac{m_0}{m_d^2}\cos x^1 )  \fp{}{x^1} - ( 	\frac{m_0}{m_d^2}\cos x^1)  \fp{}{x^2}$.
It can be checked that $g$ and $ad_eg$ are independent for $x^1\neq \pm \frac{\pi}{2}$, which corresponds to the horizontal position of the pendulum, therefore (MF1)' is satisfied everywhere except {for} $x^1= \pm \frac{\pi}{2}$. Next, we verify condition (MF2)' by calculating $\nabla_gg=\nabla_{ad_eg}g=0\in \mathcal{E}^0$. Finally, a direct calculation shows
\begin{align*}
		\nabla^2_{g,ad_eg} \, ad_eg&=\nabla^2_{ad_eg,g} \, ad_eg=\\
&	= ( \frac{m_0^2}{m_d^5} \cos^2 x^1) \fp{}{x^1} - (\frac{m_0^2}{m_d^5} \cos^2 x^1 ) \fp{}{x^2},
\end{align*}
thus $\nabla^2_{g,ad_eg} \, ad_eg-\nabla^2_{ad_eg,g} \, ad_eg=0\in \mathcal{E}^0$ satisfies (MF5)'.
The system is MF-linearizable. A linearizing function is $h(x)=\frac{m_d+J_2}{J_2}x^1+x^2$ (all others giving  MF-linearization are of the form $\sigma\, h(x)$, where $\sigma \in \mathbb{R} \backslash  \left\lbrace 0\right\rbrace $). {Due to the proof of Theorem \ref{thm:MFn2},} the linearizing diffeomorphism is $(\ti{x},\ti{y})=\Phi(x,y)=(\phi(x),D\phi(x) y)$ with $\phi(x)=(h,L_eh)^T$. The system in new coordinates reads
\begin{align}
	\label{eq:iwp-u}
	\dot{\tilde{x}}^1&=\frac{m_d+J_2}{J_2}y^1+y^2=\tilde{y}^1\nonumber\\
	\dot{\tilde{y}}^1&=\frac{m_d+J_2}{J_2}\left(\frac{m_0}{m_d}\sin x^1 - \frac{1}{m_d}u \right) -\frac{m_0}{m_d} \sin x^1 + \frac{m_d +J_2}{m_2 J_2}u\nonumber\\
	&=\frac{m_0}{J_2}\sin x^1=L_eh=\tilde{x}^2\\
	\dot{\tilde{x}}^1&=\frac{m_0}{J_2}\cos x^1 y^1=\tilde{y}^2\nonumber\\
	\dot{\tilde{y}}^2&=-\frac{m_0}{J_2} \sin x^1 y^1y^1+\frac{m_0^2}{2m_dJ_2}\sin(2x^1)- \frac{m_0}{m_d J_2}\cos x^1u=\tilde{u}\nonumber.
\end{align}

\textbf{Example 3:} We will study MF-linearizability of the TORA3 system (see Figure \ref{fig:tora}), which is based on the TORA system (Translational Oscillator with Rotational Actuator) studied in the literature, e.g. \cite{tora} (however we add gravitational effects).
\begin{figure}[h!]
	\centering
	\includegraphics[width=0.8\linewidth]{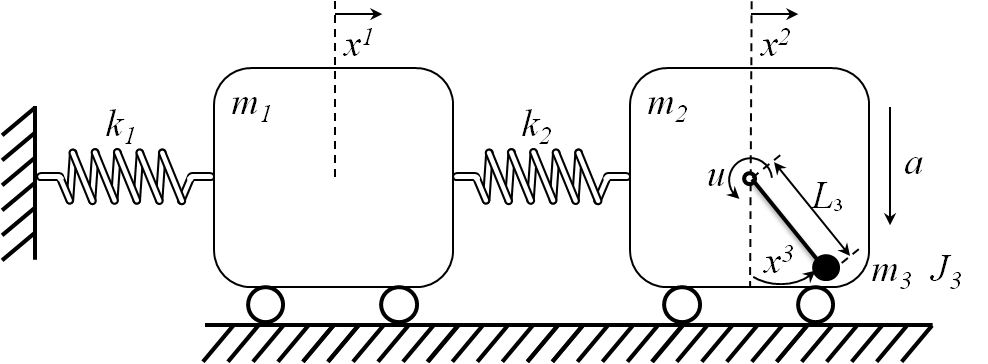}
	\caption{The TORA3 system}
	\label{fig:tora}
\end{figure}
 It consists of a two dimensional spring-mass system, with masses $m_1, m_2$ and spring constants $k_1, k_2$, respectively. A pendulum of length $l_3$, mass $m_3$, and moment of inertia $J_3$ is added to the second body. The displacements of the bodies are denoted by $x^1$ and $x^2$, respectively, and the angle of the pendulum by $x^3$. {The gravitational constant is $a$ and} $u$ is a torque applied to the pendulum. 
 The kinetic energy is
 \begin{align*}
 	T=&\frac{1}{2}m_1(\dot{x}^1)^2+\frac{1}{2}(m_2+m_3)(\dot{x}^2)^2\\&+\frac{1}{2}(J_3+m_3l_3^2)(\dot{x}^3)^2+m_3l_3\cos x^3 \dot{x}^2\dot{x}^3,
 \end{align*}
 {and the mass matrix depends on the configurations.} The potential energy is ${V=\frac{1}{2}k_1(x^1)^2 + \frac{1}{2}k_2(x^2-x^1)^2 -m_3 l_3 a \cos x^3}$.
 The equations of the dynamics read
\begin{align*}
	m_1\ddot{x}^1 + k_1 x^1 - k_2\left(x^2-x^1 \right)&=0\\ 
	(m_2+m_3)\ddot{x}^2+m_3l_3\cos x^3 \ddot{x}^3-m_3 l_3 \sin x^3 (\dot{x}^3)^2\\+k_2\left( x^2-x^1\right)  &=0\\
	m_3l_3\cos x^3 \ddot{x}^2+(m_3l_3^2+J_3)\ddot{x}^3+m_3 l_3 a \sin x^3&=u,
\end{align*}
which can be rewritten on $\T Q$ as
\begin{align}
	\label{eq:tora}
	\begin{split}
		&\dot{x}^1=y^1\qquad \dot{y}^1=\eta^1\\
		&\dot{x}^2=y^2\qquad \dot{y}^2=-\bar{\Gamma}^2_{33}y^3 y^3+\eta^2+\tau^2 u\\
		&\dot{x}^3=y^3\qquad \dot{y}^3=-\bar{\Gamma}^3_{33}y^3 y^3+\eta^3+\tau^3 u
	\end{split}
\end{align}
where ${\bar{\Gamma}^2_{33}=\frac{-\nu_0 \sin x^3}{\nu_1+\nu_2 \sin^2 x^3}}$, ${\bar{\Gamma}^3_{33}=\frac{ \nu_2 \sin x^3 \cos x^3}{\nu_1+\nu_2 \sin^2 x^3}}$, ${\eta^1=-\frac{k_1}{m_1}x^1+\frac{k_2}{m_3} \left(x^2-x^1 \right)}$, ${\eta^2=\frac{\frac{1}{2}\nu_2 a \sin 2x^3-\nu_3(x^2-x^1)}{\nu_1+\nu_2 \sin^2 x^3}}$, $\eta^3=\frac{\nu_4 \left(x^2-x^1 \right)\cos x^3-\nu_5 \sin x^3 }{\nu_1+\nu_2 \sin^2 x^3}$, ${ \tau^2=\frac{-m_3 l_3 \cos x^3}{\nu_1+\nu_2 \sin^2 x^3}}$, ${\tau^3=\frac{m_2+m_3}{\nu_1+\nu_2 \sin^2 x^3}}$, 
with constant parameters: \\
$\nu_0=m_3 l_3 (m_3l_3^2+J_3 ), \ \nu_1=m_2m_3l_3^2+J_3(m_2+m_3)$,
$\nu_2=m_3^2l_3^2,\
\nu_3=k_2\left( m_3l_3^2+J_3\right), \
\nu_4=m_3l_3 k_2$
$\nu_5=m_3 l_3 a (m_2+m_3)$.

To simplify calculations we {apply to the system} a preliminary mechanical feedback\footnote{{This preliminary feedback is not necessary and it is possible to check the conditions and to linearize the system without it, since our method and conditions are feedback invariant.}}
$
	u=\frac{1}{\tau^3}\left(\bar{\Gamma}^3_{33}y^3y^3-\eta^3+ \bar{u} \right)
$ 
which yields
\begin{align}
	\label{eq:toras}
	\begin{split}
		&\dot{x}^1=y^1\\
		&\dot{x}^2=y^2\\
		&\dot{x}^3=y^3
	\end{split}
	\begin{split}
		& \dot{y}^1=-\mu_1 x^1 + \mu_2 x^2\\
		& \dot{y}^2=\mu_3 \sin x^3 y^3 y^3 + \mu_4 (x^1-x^2 )-\mu_3 \cos x^3 u  \\
		& \dot{y}^3=\bar{u},
	\end{split}
\end{align}
with $\mu_1=\frac{k_1+k_2}{m_1}$, 
	$\mu_2=\frac{k_2}{m_1}$, 
	$\mu_3=\frac{m_3 l_3}{m_2+m_3}$,
	$\mu_4=\frac{k_2}{m_2+m_3}$.

Since conditions (MF1)-(MF4) of Theorem \ref{thm:MFC} are MF-invariant, we will check them for system \eqref{eq:toras}. To summarize:
\begin{align*}
	&	\Gamma^2_{33}=-\mu_3 \sin x^3, \quad \text{and} \quad \Gamma^i_{jk}=0 \quad  \text{otherwise}, \\
	&	e= \left( -\mu_1 x^1 + \mu_2 x^2\right) \fp{}{x^1} + \mu_4 \left(x^1-x^2 \right) \fp{}{x^2}\\
	&	g=-\mu_3 \cos x^3 \fp{}{x^2} + \fp{}{x^3}= g^2 \fp{}{x^2} + \fp{}{x^3}.
\end{align*}
We have (notice that calculations are performed on $Q$ only)
\begin{align*}
	ad_eg&=\left( \mu_2 \mu_3 \cos x^3 \right) \fp{}{x^1}-\left( \mu_3 \mu_4 \cos x^3 \right) \fp{}{x^2},\\
		ad_e^2g&= \mu_3\cos x^3\left(  \left( \mu_1 \mu_2+\mu_2\mu_4 \right) \fp{}{x^1}-\left( \mu_2 \mu_4+\mu_4^2 \right)\! \fp{}{x^2}\right) ,
\end{align*}
therefore $\rk \mathcal{E}^2=3$ for $x^3\neq \pm \frac{\pi}{2}$, and (MF1) is satisfied. Now 
\begin{align*}
	\left[g,ad_eg \right]= -\left( \mu_2 \mu_3 \sin x^3 \right) \fp{}{x^1}+\left( \mu_3 \mu_4 \sin x^3 \right) \fp{}{x^2}\in \mathcal{E}^1
\end{align*}
and (MF2) is satisfied. Then, for any vector field $v=v^i(x)\fp{}{x^i}$, 
\begin{align*}
	\nabla_vg&=\left(\fp{g^2}{x^3}+\Gamma^2_{33} \right) v^3\fp{}{x^2}=0,
\end{align*}
thus (MC3) is satisfied if we replace $v$ by, in particular, $g, ad_eg, ad_e^2g$. Finally, for (MF4), we calculate
\begin{align*}
	&	\nabla^2_{g,g}e=\left( \mu_2 \mu_3 \sin x^3 \right) \fp{}{x^1}-\left( \mu_3 \mu_4 \sin x^3 \right) \fp{}{x^2} \in \mathcal{E}^1,\\
	&	\nabla^2_{ad_e^kg,ad_e^jg}e=0 \quad \text{otherwise},
\end{align*}
thus, the system is MF-linearizable. 
Now, {choose} $h=\frac{\mu_4}{\mu_2} x^1 +x^2 +\mu_3 \sin x^3$ (whose differential $\mathrm{d}h$ annihilates $g$ and $ad_eg$), thus we take a linearizing diffeomorphism $\left( \ti{x},\ti{y}\right)=\left(\phi(x),\fp{\phi}{x}(x)y \right)  $, with $\phi(x)=\left(h,L_eh,L_e^2h \right)^T$. {The linearized system is in the form of \eqref{eq:lms} and reads}
\begin{align*}
	&\dot{\ti{x}}^1=\frac{\mu_4}{\mu_2} y^1 +y^2 +\mu_3 \cos x^3 y^3=\ti{y}^1\\
	&\dot{\ti{y}}^1=\frac{\mu_4}{\mu_2} \dot{y}^1\!+\!\dot{y}^2\!+\!\mu_3( \cos x^3 \dot{y}^3\!\!-\! \sin x^3 y^3 y^3)\!=\!\frac{\mu_4(\mu_2-\mu_1)}{\mu_2}x^1\!=\!\ti{x}^2\\
	&\dot{\ti{x}}^2=\frac{\mu_4(\mu_2-\mu_1)}{\mu_2}y^1=\ti{y}^2\\
	&\dot{\ti{y}}^2=\frac{\mu_4(\mu_2-\mu_1)}{\mu_2}\dot{y}^1=\frac{\mu_4(\mu_2-\mu_1)}{\mu_2}\left(\mu_2 x^2-\mu_1 x^1 \right)=\ti{x}^3\\
	&\dot{\ti{x}}^3=\frac{\mu_1\mu_4(\mu_1-\mu_2)}{\mu_2}y^1+\mu_4(\mu_2-\mu_1)y^2=\ti{y}^3\\
	&\dot{\ti{y}}^3=(\mu_2-\mu_1)\mu_3\mu_4\sin x^3 y^3 y^3 - \frac{(\mu_1-\mu_2)(\mu_1^2+\mu_2\mu_4)\mu_4}{\mu_2}x^1\\&\quad+ (\mu_1-\mu_2)(\mu_1+\mu_4)\mu_4 x^2+ (\mu_1-\mu_2)\mu_3\mu_4 \cos x^3 u=\ti{u}.
\end{align*}

\section{Conclusions}
\label{sec:con}
In this paper, we consider MF-linearization of mechanical control systems \eqref{eq:ms} with scalar control. We formulate the problem as a particular case of feedback linearization preserving the mechanical structure of \eqref{eq:ms} so that the transformed system is both linear and mechanical. {As we showed in \cite{NR} and confirmed in this paper, even in the simplest case, the class of MF-linearizable systems is substantially smaller than that of general F-linearizable ones. Therefore, a natural question arises, namely to compare the conditions presented in this paper with those for F-linearization. The answer lies in the interplay between the distributions $\mathcal{E}^i=\spn{ad_e^jg, 0\leq j\leq i }$ and the "usual" for F-linearization  $	\mathcal{D}^i=\spn{ad_F^jG, 0\leq j\leq i }$. We will address this problem in the future. }

\section{Appendix}
\label{sec:apx}
The following lemma can be proved by a direct calculation.
\begin{lemma}
	\label{prop:cov2}
	The second covariant derivative $\nabla^2_{X,Y}Z$ satisfies the following properties:	
	\begin{enumerate}[(i)]
		\item linearity over $C^{\infty}(Q)$ in $X$ and $Y$:
		\begin{align*}
			\nabla^2_{(\alpha_1X_1+\alpha_2X_2),Y}Z&=\alpha_1\nabla^2_{X_1,Y}Z+\alpha_2 \nabla^2_{X_2,Y}Z\\
			\nabla^2_{X,(\alpha_1Y_1+\alpha_2Y_2)}Z&=\alpha_1\nabla^2_{X,Y_1}Z+\alpha_2 \nabla^2_{X,Y_1}Z
		\end{align*} 
		\item linearity over $\mathbb{R}$ in $Z$:
		\begin{align*}
			\nabla^2_{X,Y}(a_1Z_1+a_2Z_2)=a_1\nabla^2_{X,Y}Z_1+a_2\nabla^2_{X,Y}Z_2
		\end{align*}
		\item the product rule:
		\begin{align*}
			\nabla^2_{X,Y}(\beta Z)= &\beta \nabla^2_{X,Y}Z + L_X\beta \nabla_YZ\\&+L_Y\beta \nabla_XZ+ \left( \nabla^2_{X,Y}\beta\right)  Z,
		\end{align*}
	\end{enumerate}
	where $\nabla^2_{X,Y}\beta=L_XL_Y\beta-L_{\nabla_XY}\beta\in C^{\infty}(Q)$,
	$X_i,Y_i,Z_i\in \mathfrak{X}(Q)$, $\alpha_i,\beta\in C^{\infty}(Q)$, and $a_i\in \mathbb{R}$.
\end{lemma}

{The following lemma is crucial for the proof of Theorem~\ref{thm:MFC}.}
{
\begin{lemma}
	\label{lem:formula}
	For the system 
		\begin{align}
		\label{eq:nft1}
		\begin{split}
			\dot{{x}}^{1}&={y}^{1}\\
			\dot{{x}}^{i}&={y}^{i}
		\end{split}
		\begin{split}
			\dot{{y}}^{1}&={u}\\
			\dot{{y}}^{i}&=-{\Gamma}^{i}_{jk}{y}^j{y}^k +{x}^{i-1}, \quad 2 \leq i \leq n,
		\end{split}
	\end{align}
 we have for any $1\leq k,j \leq n $,
\begin{align}
	\label{eq:formula}
		&\nabla^2_{ad_{{e}}^{k-1}{g},ad_{{e}}^{j-1}{g}}{e}=(-1)^{j+k}\bigg( \fp{{\Gamma}^i_{j s}}{{x}^k} {e}^s + {\Gamma}^i_{jk+1} +{\Gamma}^i_{kj+1}	- {\Gamma}^{i-1}_{kj}\nonumber \\ &\qquad \quad+({\Gamma}^d_{js}{\Gamma}^i_{kd}- {\Gamma}^d_{kj}{\Gamma}^i_{ds}){e}^s 
\bigg) \fp{}{{x}^i}.
\end{align}
\end{lemma}
\begin{proof}
For system \eqref{eq:nft1} we calculate $	\nabla^2_{ad_{{e}}^{k-1}{g},ad_{{e}}^{j-1}{g}}{e}$
${=(-1)^{j+k}\nabla^2_{\fp{}{{x}^k},\fp{}{{x}^{j}}}{e}}= \nabla_{\fp{}{{x}^k}}\nabla_{\fp{}{{x}^{j}}}{e}-\nabla_{\nabla_{\fp{}{{x}^k}}\fp{}{{x}^{j}}}{e}$,\\
where
$
\nabla_{\fp{}{{x}^{j}}}{e}=\left(\fp{{e}^d}{{x}^{j}}+{\Gamma}^d_{js}{e}^s \right)\fp{}{{x}^d}
$, and 	
\begin{align*}
	&\nabla_{\fp{}{{x}^k}}\left(\nabla_{\fp{}{{x}^{j}}}{e} \right)= \nabla_{\fp{}{{x}^k}}\left( \fp{{e}^d}{{x}^{j}}\right) \fp{}{{x}^d}+
	\nabla_{\fp{}{{x}^k}}\left({\Gamma}^d_{js}{e}^s \right)\fp{}{{x}^d}\\
	&=\fp{{e}^d}{{x}^{j}}\nabla_{\fp{}{{x}^k}}\fp{}{{x}^d}+L_{\fp{}{{x}^k}}\left( \fp{{e}^d}{{x}^{j}}\right) \fp{}{{x}^d}+ \left( {\Gamma}^d_{js}{e}^s\right) \nabla_{\fp{}{{x}^k}}\fp{}{{x}^d}\\
	&+\left( L_{\fp{}{{x}^k}}\left( {\Gamma}^d_{js}\right){e}^s+ L_{\fp{}{{x}^k}}\left( {e}^s\right){\Gamma}^d_{js} \right)\fp{}{{x}^d} \\
	&= \fp{{e}^d}{{x}^{j}}{\Gamma}^i_{kd} \fp{}{{x}^i} + {\Gamma}^d_{j s} {e}^s {\Gamma}^i_{kd}\fp{}{{x}^i} + \left( \fp{{\Gamma}^i_{j s}}{{x}^k} {e}^s+ \fp{{e}^s}{{x}^k}{\Gamma}^i_{js}\right) \fp{}{{x}^i}\\
	&=\left( \fp{{\Gamma}^i_{js}}{{x}^k} {e}^s + {\Gamma}^i_{jk+1} +{\Gamma}^i_{kj+1} + {\Gamma}^d_{js}{\Gamma}^i_{kd} {e}^s\right) \fp{}{{x}^i},
\end{align*}
since $\fp{e^d}{x^j}=1$, if $d=j+1$, and zero otherwise, and thus $\fp{e^d}{x^j}\Gamma^i_{kd}=\Gamma^i_{kj+1}$ (analogously for the other derivatives).
Now, using $\nabla_{\fp{}{{x}^k}}\fp{}{{x}^{j}}={\Gamma}^d_{kj}\fp{}{{x}^d}$, we calculate
\begin{align*}
	&\nabla_{\nabla_{\fp{}{{x}^k}}\fp{}{{x}^{j}}}{e}=\nabla_{{\Gamma}^d_{kj}\fp{}{{x}^d}}{e}={\Gamma}^d_{kj}\left(\fp{{e}^i}{{x}^d}+{\Gamma}^i_{ds}{e}^s \right) \fp{}{{x}^i}\\&=\left({\Gamma}^{i-1}_{kj}+{\Gamma}^d_{kj}{\Gamma}^i_{ds}{e}^s \right)\fp{}{{x}^i}, \qquad \text{ so we have}
\end{align*}
\begin{align*}
	\begin{split}
		&\nabla^2_{\fp{}{{x}^k},\fp{}{{x}^{j}}}{e}=\nabla_{\fp{}{{x}^k}}\left(\nabla_{\fp{}{{x}^{j}}}{e} \right)-\nabla_{\nabla_{\fp{}{{x}^k}}\fp{}{{x}^{j}}}{e}\\
		&=\bigg( \fp{{\Gamma}^i_{j s}}{{x}^k} {e}^s + {\Gamma}^i_{jk+1} +{\Gamma}^i_{kj+1}  - {\Gamma}^{i-1}_{kj}\\ &\qquad \quad+({\Gamma}^d_{js}{\Gamma}^i_{kd}- {\Gamma}^d_{kj}{\Gamma}^i_{ds}){e}^s 
		\bigg) \fp{}{{x}^i}.
	\end{split}
\end{align*}
which yields \eqref{eq:formula}.
\end{proof}
}

\end{document}